\documentclass[10pt]{article}
\usepackage{amsmath, amsthm, amssymb, verbatim, tikz, xcolor, subfigure}
\usepackage{latexsym}
\usepackage{amsfonts}
\usepackage{graphicx}

\usepackage{float}
\usepackage{bigints}
\usepackage{fullpage}

\usepackage{footmisc}
\DeclareMathOperator{\vol}{vol}

\newcommand{\Eu}{\mathbb{E}}

\newcommand{\BB}{\mathbf{B}}

\renewcommand{\P}{\mathcal{P}}
\renewcommand{\c}{\mathbf{c}}

\begin{document}

\newtheorem{theorem}{Theorem}[section]
\newtheorem{question}[theorem]{Question}
\newtheorem{example}[theorem]{Example}
\newtheorem{observation}[theorem]{Observation}
\newtheorem{remark}[theorem]{Remark}
\newtheorem{fact}{Fact}
\newtheorem{proc}[theorem]{Procedure}
\newtheorem{conjecture}[theorem]{Conjecture}
\newtheorem{lemma}[theorem]{Lemma}
\newtheorem{proposition}[theorem]{Proposition}
\newtheorem{corollary}[theorem]{Corollary}
\newtheorem{fac}[theorem]{Fact}
\newtheorem*{abst}{Abstract}
\newtheorem{definition}{Definition}
\newtheorem{problem}{Problem}
\newtheorem{algo}{Algorithm}

\title{Contact numbers for sphere packings} 

\author{K\'{a}roly Bezdek and Muhammad A. Khan}
\date{}
 \maketitle

\begin{abstract}
In discrete geometry, the contact number of a given finite number of non-overlapping spheres was introduced as a generalization of Newton's kissing number. This notion has not only led to interesting mathematics, but has also found applications in the science of self-assembling materials, such as colloidal matter. With geometers, chemists, physicists and materials scientists researching the topic, there is a need to inform on the state of the art of the contact number problem. In this paper, we investigate the problem in general and emphasize important special cases including contact numbers of minimally rigid and totally separable sphere packings. We also discuss the complexity of recognizing contact graphs in a fixed dimension. Moreover, we list some conjectures and open problems.           
\vspace{2mm}


\vspace{2mm}

\end{abstract}

\section{Introduction}\label{intro}

The well-known {\it ``kissing number problem"} asks for the maximum number $k(d)$ of non-overlapping unit balls that can touch a unit ball in the $d$-dimensional Euclidean space $\mathbb{E}^{d}$. The problem originated in the 17th century from a disagreement between Newton and Gregory about how many 3-dimensional unit spheres without overlap could touch a given unit sphere. The former maintained that the answer was 12, while the latter thought it was 13. The question was finally settled many years later \cite{ScWa} when Newton was proved correct. The known values of $k(d)$ are $k(2)=6$ (trivial), $k(3)=12$ (\cite{ScWa}), $k(4)=24$ (\cite{Mu}), $k(8)=240$ (\cite{OdSl}), and $k(24)=196560$ (\cite{OdSl}). The problem of finding kissing numbers is closely connected to the more general problems of finding bounds for spherical codes and sphere packings. For old and new results on kissing numbers we refer the interested reader to the recent survey article \cite{BoDoMu}. In this paper, we focus on a more general relative of kissing number called contact number.

Let $\mathbf{B}^d$ be the $d$-dimensional unit ball centered at the origin ${\bf o}$ in $\mathbb{E}^{d}$. As is well known, a {\it finite packing} of unit balls in $\mathbb{E}^{d}$ is a finite family of non-overlapping translates of $\mathbf{B}^d$ in $\mathbb{E}^{d}$. Furthermore, the {\it contact graph} of a finite unit ball packing in $\mathbb{E}^{d}$ is the (simple) graph whose vertices correspond to the packing elements and whose two vertices are connected by an edge if and only if the corresponding two packing elements touch each other. The number of edges of a contact graph is called the {\it contact number} of the underlying unit ball packing. The {\it ``contact number problem"} asks for the largest contact number, that is, for the maximum number $c(n,d)$ of edges that a contact graph of $n$ non-overlapping translates of $\mathbf{B}^d$ can have in $\mathbb{E}^{d}$. 

The problem of determining $c(n,d)$ is equivalent to Erd\H{o}s's repeated shortest distance problem, which asks for the largest number of repeated shortest distances among $n$ points in ${\mathbb{E}}^{d}$. The planar case of this question was originally raised by Erd\H{o}s in 1946 \cite{ErDist}, with an answer conjectured by Reutter in 1972 and established by Harborth \cite{Ha} in 1974, whereas the problem in its more general forms was popularized by Erd\H{o}s and Ulam. Another way to look at the contact number problem is to think of it as the combinatorial analogue of the densest sphere packing problem, which dates back to the 17th century.

Let ${\bf K}$ be a convex body, i.e., a compact convex set with non-empty interior in $\mathbb{E}^{d}$. (If $d=2$, then  ${\bf K}$ is called a convex domain.) If ${\bf K}$ is symmetric about the origin ${\bf o}$ in $\mathbb{E}^{d}$, then one can regard ${\bf K}$ as the unit ball of a given norm in $\mathbb{R}^{d}$. In the same way as above one can talk about the largest contact number of packings by $n$ translates of ${\bf K}$ in $\mathbb{E}^{d}$ and label it by $c({\bf K},n,d)$. Here we survey the results on $c(n,d)$ as well as $c({\bf K},n,d)$.    

The notion of total separability was introduced in \cite{FTFT73} as follows: a packing of unit balls in $\mathbb{E}^{d}$ is called {\it totally separable} if any two unit balls can be separated by a hyperplane of $\mathbb{E}^{d}$ such that it is disjoint from the interior of each unit ball in the packing. Finding the densest totally separable unit ball packing is a difficult problem, which is solved only in dimensions two (\cite{FTFT73}, \cite{BA83}) and three (\cite{K88}). As a close combinatorial relative it is natural to investigate the maximum contact number $c_{\rm sep}(n,d)$ of totally separable packings of $n$ unit balls in $\mathbb{E}^{d}$. In what follows, we survey the results on $c_{\rm sep}(n,d)$ as well.
 
The paper is organized as follows. In Section \ref{materials}, we briefly discuss the importance of the contact number problem in materials science. The next two sections are devoted to the known bounds on the contact number for $d=2,3$. Section \ref{sec:empirical}, explores three computer-assisted empirical approaches that have been developed by applied scientists to estimate the contact numbers of packings of small number of unit spheres in ${\mathbb{E}}^{3}$. We analyze these approaches at length and show that despite being of interest, they fall short of providing exact values of largest contact numbers. In Section \ref{sec:polyomino}, we study contact numbers of unit sphere packings in ${\mathbb{E}}^{2}$ and ${\mathbb{E}}^{3}$ that live on the integer lattice or are totally separable. Section \ref{sec:d-dim} covers recent general results on packings of congruent balls and translates of an arbitrary convex body in $d$-space. It also includes results on the integer lattice and totally separable packings of $d$-dimensional unit balls. Finally, the last section deals with the state of the contact number problem for non-congruent sphere packings.

\section{Motivation from materials science}\label{materials}

In addition to finding its origins in the works of pioneers like Newton, Erd\H{o}s, Ulam and Fejes T\'{o}th (see Section \ref{sec:plane} for more on the role of latter two), the contact number problem is also important from an applications point of view. Packings of hard sticky spheres - impenetrable spheres with short-range attractive forces - provide excellent models for the formation of several real-life materials such as colloids, powders, gels and glasses \cite{Hay}. The particles in these materials can be thought of as hard spheres that self-assemble into small and large clusters due to their attractive forces. This process,  called \textit{self-assembly}, is of tremendous interest to materials scientists, chemists, statistical physicists and biologists alike. 

Of particular interest are \textit{colloids}, which consist of particles at micron scale, dispersed in a fluid and kept suspended by thermal interactions \cite{Ma}. Colloidal matter occurs abundantly around us - for example in glue, milk and paint. Moreover, controlled colloid formation is a fundamental tool used in scientific research to understand the phenomena of self-assembly and phase transition. 

From thermodynamical considerations it is clear that colloidal particles assemble so as to minimize the potential energy of the cluster. Since the range of attraction between these particles is extremely small compared to their sizes, two colloidal particles do not exert any force on each other until they are infinitessimally close, at which point their is strong attraction between them. As a result, they stick together, are resistant to drift apart, but strongly resistant to move any closer \cite{ArMaBr, Hay}. Thus two colloidal particles experiencing an attractive force from one another in a cluster can literally be thought of as being in contact. 

It can be shown that under the force law described above, the potential energy of a colloidal cluster at reasonably low temperatures is inversely proportional to the number of contacts between its particles \cite{ArMaBr, HC, Hoy}. Thus the particles are highly likely to assemble in packings that maximize the contact number. This has generated significant interest among materials scientists towards the contact number problem \cite{ArMaBr, Hoy} and has led to efforts in developing computer-assisted approaches to attack the problem. More details will appear in Section \ref{sec:empirical}.

\section{Largest contact numbers in the plane}\label{sec:plane}
\subsection{The Euclidean plane}
Harborth \cite{Ha} proved the following well-known result on the contact graphs of congruent circular disk packings in $\mathbb{E}^{2}$. 

\begin{theorem}\label{Harborth}
$c(n,2)=\lfloor 3n-\sqrt{12n-3}\rfloor$, for all $n\ge 2$.
\end{theorem}

This result shows that an optimal way to pack $n$ congruent disks to maximize their contacts is to pack them in a `hexagonal arrangement'. The arrangement starts by packing 6 unit disks around a central disk in such a way that the centers of the surrounding disks form a regular hexagon. The pattern is then continued by packing hexagonal layers of disks around the first hexagon. Thus the hexagonal packing arrangement, which is known to be the densest congruent disk packing arrangement, also achieves the maximum contact number $c(n,2)$, for all $n$.    

Interestingly, this also means that $c(n,2)$ equals the maximum number of sides that can be shared between $n$ cells of a regular hexagon tiling of the plane. This connection was explored in \cite{HaHa}, where isoperimetric hexagonal lattice animals of a given area $n$ were explored. The connection between contact numbers and isoperimetric lattice animals is studied in detail in Section \ref{sec:polyomino}. So we skip the details here. 

Despite the existence of a simple formula for $c(n, 2)$, recognizing contact graphs of congruent disk packings is a challenging problem. The difficulty of this problem is made apparent by the following complexity result from \cite{BrKi}.

\begin{theorem}\label{2-hard}
The problem of recognizing contact graphs of unit disk packings is NP-hard.
\end{theorem}

Quite surprisingly, the following rather natural stability version of Theorem~\ref{Harborth} is still an open problem. (See also the final remarks in \cite{Br1}.) 

\begin{conjecture}
There exists an $\epsilon>0$ such that for any packing of $n$ circular disks of radii chosen from the interval $[1-\epsilon, 1]$ the number of touching pairs in the packing is at most $\lfloor 3n-\sqrt{12n-3}\rfloor$, for all $n\ge 2$.
\end{conjecture}

In 1984, Ulam (\cite{Er}) proposed to investigate Erd\H os-type distance problems in normed spaces. Pursuing this idea, Brass \cite{Br1} proved the following extension of Theorem~\ref{Harborth} to normed planes.

\begin{theorem}\label{Brass}
Let ${\bf K}$ be a convex domain different from a parallelogram in $\mathbb{E}^{2}$. Then for all $n\geq 2$, one has $c({\bf K},n,2)=\lfloor 3n-\sqrt{12n-3}\rfloor$. If ${\bf K}$ is a parallelogram, then $c({\bf K},n,2)=\lfloor 4n-\sqrt{28n-12}\rfloor$ holds for all $n\ge 2$. 
\end{theorem}

The same idea inspired the first named author to investigate this question in $d$-space, details of which appear in Section \ref{sec:d-dim}. 

Returning to normed planes, the following is a natural question.

\begin{problem}
Find an analogue of Theorem~\ref{Brass} for totally separable translative packings of convex domains in $\mathbb{E}^{2}$.
\end{problem}

\subsection{Spherical and hyperbolic planes}

An analogue of Harborth's theorem in the hyperbolic plane $\mathbb{H}^{2}$ was found by Bowen in \cite{Bo}. In fact, his method extends to the $2$-dimensional spherical plane $\mathbb{S}^{2}$. We prefer to quote these results as follows.

\begin{theorem}
Consider disk packings in $\mathbb{H}^{2}$ (resp., $\mathbb{S}^{2}$) by finitely many congruent disks, which maximize the number of touching pairs for the given number of congruent disks and of given diameter $D$. Then such a packing must have all of its centers located on the vertices of a triangulation of  $\mathbb{H}^{2}$ (resp., $\mathbb{S}^{2}$) by congruent equilateral triangles of side length $D$ provided that the equilateral triangle in $\mathbb{H}^{2}$ (resp., $\mathbb{S}^{2}$) of side length $D$ has each of its angles equal to $\frac{2\pi}{N}$ for some positive integer $N\geq 3$.
\end{theorem}

In 1984, L. Fejes T\' oth (\cite{BeCoKe}) raised the following attractive and related problem in $\mathbb{S}^{2}$: Consider an arbitrary packing ${\mathcal P}_r$ of disks of radius $r>0$ in $\mathbb{S}^{2}$. Let ${\rm deg}_{\rm avr}({\mathcal P}_r)$ denote the average degree of the vertices of the contact graph of ${\mathcal P}_r$. Then prove or disprove that $\limsup_{r\to 0}\left(\sup_{{\mathcal P}_r} {\rm deg}_{\rm avr}({\mathcal P}_r)\right)<5$. This problem was settled in \cite{BeCoKe}.

\begin{theorem}\label{Bezdek-Connelly-Kertesz}
Let ${\mathcal P}_r$ be an arbitrary packing of disks of radius $r>0$ in $\mathbb{S}^{2}$. Then 
$$\limsup_{r\to 0}\left(\sup_{{\mathcal P}_r} {\rm deg}_{\rm avr}({\mathcal P}_r)\right)<5.$$
\end{theorem}

We conclude this section with the still open hyperbolic analogue of Theorem~\ref{Bezdek-Connelly-Kertesz} which was raised in \cite{BeCoKe}.

\begin{conjecture}\label{B-C-K}
Let ${\mathcal P}_r$ be an arbitrary packing ${\mathcal P}_r$ of disks of radius $r>0$ in $\mathbb{H}^{2}$. Then 
$$\limsup_{r\to 0}\left(\sup_{{\mathcal P}_r} {\rm deg}_{\rm avr}({\mathcal P}_r)\right)<5.$$
\end{conjecture}

\section{Largest contact numbers in $3$-space}\label{sec:space}

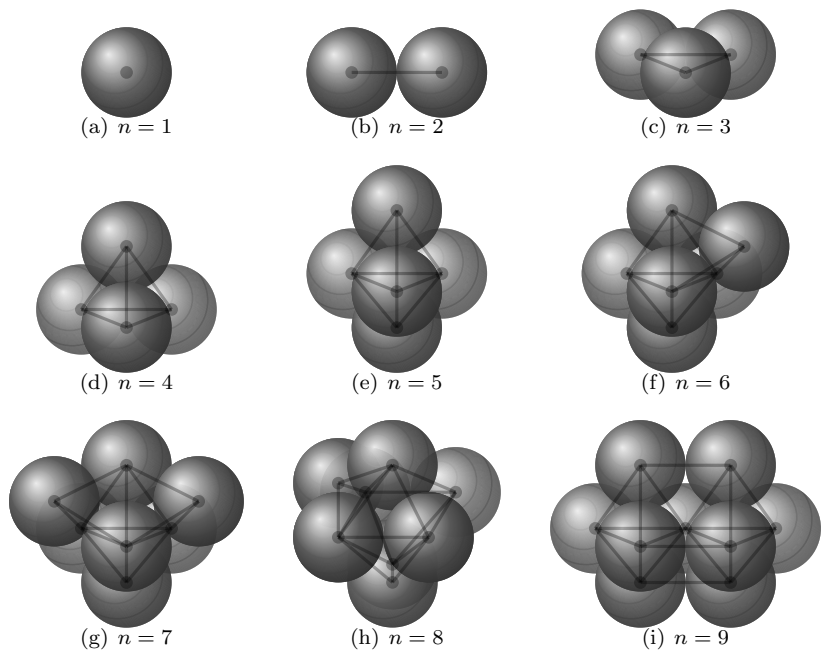
\begin{figure}[ht]\label{fig:maxcontact}
\centering
\begin{tabular}{ccc}
\subfigure [$n=1$]{
\begin{tikzpicture}[scale=1.2]
\centering
\shade[ball color=black!50!white, opacity=0.7] (9,.5) circle (.5cm);

\fill[opacity=0.3] (9,0.5) circle (2pt); 

\end{tikzpicture}
} &

\subfigure [$n=2$]{
\begin{tikzpicture}[scale=1.2]
\centering
\shade[ball color=black!50!white, opacity=0.7] (9,.5) circle (.5cm);
\shade[ball color=black!50!white, opacity=0.7] (10,.5) circle (.5cm);

\fill[opacity=0.3] (9,0.5) circle (2pt); 
\fill[opacity=0.3] (10,0.5) circle (2pt); 

\draw[very thick, opacity=0.3] (9,0.5) -- (10,0.5); 
\end{tikzpicture}
} &

\subfigure [$n=3$]{
\begin{tikzpicture}[scale=1.2]
\centering
\shade[ball color=black!50!white, opacity=0.5] (9,.5) circle (.5cm);
\shade[ball color=black!50!white, opacity=0.5] (10,.5) circle (.5cm);
\shade[ball color=black!50!white, opacity=0.7] (9.5,.3) circle (.5cm);

\fill[opacity=0.3] (9,0.5) circle (2pt); 
\fill[opacity=0.3] (10,0.5) circle (2pt); 
\fill[opacity=0.3] (9.5,0.3) circle (2pt); 

\draw[very thick, opacity=0.3] (9,0.5) -- (10,0.5); 
\draw[very thick, opacity=0.3] (9,0.5) -- (9.5,0.3); 
\draw[very thick, opacity=0.3] (10,0.5) -- (9.5,0.3); 

\end{tikzpicture}
} \\

\subfigure [$n=4$]{
\begin{tikzpicture}[scale=1.2]
\centering
\shade[ball color=black!50!white, opacity=0.4] (9,.5) circle (.5cm);
\shade[ball color=black!50!white, opacity=0.4] (10,.5) circle (.5cm);
\shade[ball color=black!50!white, opacity=0.6] (9.5,1.2) circle (.5cm);
\shade[ball color=black!50!white, opacity=0.7] (9.5,.3) circle (.5cm);

\fill[opacity=0.3] (9,0.5) circle (2pt); 
\fill[opacity=0.3] (10,0.5) circle (2pt); 
\fill[opacity=0.3] (9.5,1.2) circle (2pt); 
\fill[opacity=0.3] (9.5,0.3) circle (2pt); 

\draw[very thick, opacity=0.3] (9,0.5) -- (10,0.5); 
\draw[very thick, opacity=0.3] (9,0.5) -- (9.5,0.3); 
\draw[very thick, opacity=0.3] (10,0.5) -- (9.5,0.3); 
\draw[very thick, opacity=0.3] (9.5,1.2) -- (9.5,0.3); 
\draw[very thick, opacity=0.3] (9.5,1.2) -- (10,0.5); 
\draw[very thick, opacity=0.3] (9.5,1.2) -- (9,0.5); 
\end{tikzpicture}
} &

\subfigure [$n=5$]{
\begin{tikzpicture}[scale=1.2]
\centering
\shade[ball color=black!50!white, opacity=0.4] (9,.5) circle (.5cm);
\shade[ball color=black!50!white, opacity=0.4] (10,.5) circle (.5cm);
\shade[ball color=black!50!white, opacity=0.6] (9.5,1.2) circle (.5cm);
\shade[ball color=black!50!white, opacity=0.6] (9.5,-0.1) circle (.5cm);
\shade[ball color=black!50!white, opacity=0.7] (9.5,.3) circle (.5cm);

\fill[opacity=0.3] (9,0.5) circle (2pt); 
\fill[opacity=0.3] (10,0.5) circle (2pt); 
\fill[opacity=0.3] (9.5,1.2) circle (2pt); 
\fill[opacity=0.3] (9.5,-0.1) circle (2pt); 
\fill[opacity=0.3] (9.5,0.3) circle (2pt); 

\draw[very thick, opacity=0.3] (9,0.5) -- (10,0.5); 
\draw[very thick, opacity=0.3] (9,0.5) -- (9.5,0.3); 
\draw[very thick, opacity=0.3] (10,0.5) -- (9.5,0.3); 
\draw[very thick, opacity=0.3] (9.5,1.2) -- (9.5,0.3); 
\draw[very thick, opacity=0.3] (9.5,1.2) -- (10,0.5); 
\draw[very thick, opacity=0.3] (9.5,1.2) -- (9,0.5); 
\draw[very thick, opacity=0.3] (9.5,-0.1) -- (9.5,0.3); 
\draw[very thick, opacity=0.3] (9.5,-0.1) -- (10,0.5); 
\draw[very thick, opacity=0.3] (9.5,-0.1) -- (9,0.5); 
\end{tikzpicture}
} &

\subfigure [$n=6$]{
\begin{tikzpicture}[scale=1.2]
\centering
\shade[ball color=black!50!white, opacity=0.4] (9,.5) circle (.5cm);
\shade[ball color=black!50!white, opacity=0.4] (10,.5) circle (.5cm);
\shade[ball color=black!50!white, opacity=0.6] (9.5,1.2) circle (.5cm);
\shade[ball color=black!50!white, opacity=0.6] (9.5,-0.1) circle (.5cm);
\shade[ball color=black!50!white, opacity=0.7] (10.3,0.8) circle (.5cm);
\shade[ball color=black!50!white, opacity=0.7] (9.5,.3) circle (.5cm);

\fill[opacity=0.3] (9,0.5) circle (2pt); 
\fill[opacity=0.3] (10,0.5) circle (2pt); 
\fill[opacity=0.3] (9.5,1.2) circle (2pt); 
\fill[opacity=0.3] (9.5,-0.1) circle (2pt); 
\fill[opacity=0.3] (9.5,0.3) circle (2pt); 
\fill[opacity=0.3] (10.3,0.8) circle (2pt); 

\draw[very thick, opacity=0.3] (9,0.5) -- (10,0.5); 
\draw[very thick, opacity=0.3] (9,0.5) -- (9.5,0.3); 
\draw[very thick, opacity=0.3] (10,0.5) -- (9.5,0.3); 
\draw[very thick, opacity=0.3] (9.5,1.2) -- (9.5,0.3); 
\draw[very thick, opacity=0.3] (9.5,1.2) -- (10,0.5); 
\draw[very thick, opacity=0.3] (9.5,1.2) -- (9,0.5); 
\draw[very thick, opacity=0.3] (9.5,-0.1) -- (9.5,0.3); 
\draw[very thick, opacity=0.3] (9.5,-0.1) -- (10,0.5); 
\draw[very thick, opacity=0.3] (9.5,-0.1) -- (9,0.5); 
\draw[very thick, opacity=0.3] (10.3,0.8) -- (9.5,0.3); 
\draw[very thick, opacity=0.3] (10.3,0.8) -- (10,0.5); 
\draw[very thick, opacity=0.3] (10.3,0.8) -- (9.5,1.2); 
\end{tikzpicture}
}\\

\subfigure [$n=7$]{
\begin{tikzpicture}[scale=1.2]
\centering
\shade[ball color=black!50!white, opacity=0.4] (9,.5) circle (.5cm);
\shade[ball color=black!50!white, opacity=0.4] (10,.5) circle (.5cm);
\shade[ball color=black!50!white, opacity=0.6] (9.5,1.2) circle (.5cm);
\shade[ball color=black!50!white, opacity=0.6] (9.5,-0.1) circle (.5cm);
\shade[ball color=black!50!white, opacity=0.7] (10.3,0.8) circle (.5cm);
\shade[ball color=black!50!white, opacity=0.7] (8.7,0.8) circle (.5cm);
\shade[ball color=black!50!white, opacity=0.7] (9.5,.3) circle (.5cm);

\fill[opacity=0.3] (9,0.5) circle (2pt); 
\fill[opacity=0.3] (10,0.5) circle (2pt); 
\fill[opacity=0.3] (9.5,1.2) circle (2pt); 
\fill[opacity=0.3] (9.5,-0.1) circle (2pt); 
\fill[opacity=0.3] (9.5,0.3) circle (2pt); 
\fill[opacity=0.3] (10.3,0.8) circle (2pt); 
\fill[opacity=0.3] (8.7,0.8) circle (2pt); 

\draw[very thick, opacity=0.3] (9,0.5) -- (10,0.5); 
\draw[very thick, opacity=0.3] (9,0.5) -- (9.5,0.3); 
\draw[very thick, opacity=0.3] (10,0.5) -- (9.5,0.3); 
\draw[very thick, opacity=0.3] (9.5,1.2) -- (9.5,0.3); 
\draw[very thick, opacity=0.3] (9.5,1.2) -- (10,0.5); 
\draw[very thick, opacity=0.3] (9.5,1.2) -- (9,0.5); 
\draw[very thick, opacity=0.3] (9.5,-0.1) -- (9.5,0.3); 
\draw[very thick, opacity=0.3] (9.5,-0.1) -- (10,0.5); 
\draw[very thick, opacity=0.3] (9.5,-0.1) -- (9,0.5); 
\draw[very thick, opacity=0.3] (10.3,0.8) -- (9.5,0.3); 
\draw[very thick, opacity=0.3] (10.3,0.8) -- (10,0.5); 
\draw[very thick, opacity=0.3] (10.3,0.8) -- (9.5,1.2); 
\draw[very thick, opacity=0.3] (8.7,0.8) -- (9.5,0.3); 
\draw[very thick, opacity=0.3] (8.7,0.8) -- (9,0.5);
\draw[very thick, opacity=0.3] (8.7,0.8) -- (9.5,1.2); 
\end{tikzpicture}
} &

\subfigure [$n=8$]{
\begin{tikzpicture}[scale=1.2]
\centering
\shade[ball color=black!50!white, opacity=0.4] (9.3,1) circle (.5cm);
\shade[ball color=black!50!white, opacity=0.4] (10.3,1) circle (.5cm);
\shade[ball color=black!50!white, opacity=0.6] (9,1.1) circle (.5cm);
\shade[ball color=black!50!white, opacity=0.4] (9.6,0.2) circle (.5cm);
\shade[ball color=black!50!white, opacity=0.6] (9.6,0) circle (.5cm);
\shade[ball color=black!50!white, opacity=0.6] (9.6,1.3) circle (.5cm);
\shade[ball color=black!50!white, opacity=0.8] (9,.5) circle (.5cm);
\shade[ball color=black!50!white, opacity=0.8] (10,.5) circle (.5cm);

\fill[opacity=0.3] (9.3,1) circle (2pt); 
\fill[opacity=0.3] (10.3,1) circle (2pt); 
\fill[opacity=0.3] (9.6,1.3) circle (2pt);
\fill[opacity=0.3] (9.6,0.2) circle (2pt);
\fill[opacity=0.3] (9.6,0) circle (2pt);
\fill[opacity=0.3] (9,1.1) circle (2pt);
\fill[opacity=0.3] (9,0.5) circle (2pt); 
\fill[opacity=0.3] (10,0.5) circle (2pt); 

\draw[very thick, opacity=0.3] (9,0.5) -- (10,0.5); 
\draw[very thick, opacity=0.3] (10,0.5) -- (10.3,1); 
\draw[very thick, opacity=0.3] (10.3,1) -- (9.3,1); 
\draw[very thick, opacity=0.3] (9.3,1) -- (9,0.5); 

\draw[very thick, opacity=0.3] (9.6,1.3) -- (10,0.5); 
\draw[very thick, opacity=0.3] (9.6,1.3) -- (10.3,1); 
\draw[very thick, opacity=0.3] (9.6,1.3) -- (9.3,1); 
\draw[very thick, opacity=0.3] (9.6,1.3) -- (9,0.5); 

\draw[very thick, opacity=0.3] (9.6,0.2) -- (10,0.5); 
\draw[very thick, opacity=0.3] (9.6,0.2) -- (10.3,1); 
\draw[very thick, opacity=0.3] (9.6,0.2) -- (9.3,1); 
\draw[very thick, opacity=0.3] (9.6,0.2) -- (9,0.5); 

\draw[very thick, opacity=0.3] (9.6,0) -- (10,0.5); 
\draw[very thick, opacity=0.3] (9.6,0) -- (9.6,0.2); 
\draw[very thick, opacity=0.3] (9.6,0) -- (9,0.5); 

\draw[very thick, opacity=0.3] (9,1.1) -- (9.6,1.3); 
\draw[very thick, opacity=0.3] (9,1.1) -- (9.3,1); 
\draw[very thick, opacity=0.3] (9,1.1) -- (9,0.5); 
\end{tikzpicture}
} & 

\subfigure [$n=9$]{
\begin{tikzpicture}[scale=1.2]
\centering
\shade[ball color=black!50!white, opacity=0.4] (9,.5) circle (.5cm);
\shade[ball color=black!50!white, opacity=0.4] (10,.5) circle (.5cm);
\shade[ball color=black!50!white, opacity=0.4] (11,.5) circle (.5cm);

\shade[ball color=black!50!white, opacity=0.6] (9.5,1.2) circle (.5cm);
\shade[ball color=black!50!white, opacity=0.6] (10.5,1.2) circle (.5cm);

\shade[ball color=black!50!white, opacity=0.6] (9.5,-0.1) circle (.5cm);
\shade[ball color=black!50!white, opacity=0.6] (10.5,-0.1) circle (.5cm);

\shade[ball color=black!50!white, opacity=0.7] (9.5,.3) circle (.5cm);
\shade[ball color=black!50!white, opacity=0.7] (10.5,.3) circle (.5cm);

\fill[opacity=0.3] (9,0.5) circle (2pt); 
\fill[opacity=0.3] (10,0.5) circle (2pt); 
\fill[opacity=0.3] (11,0.5) circle (2pt);
\fill[opacity=0.3] (9.5,1.2) circle (2pt); 
\fill[opacity=0.3] (10.5,1.2) circle (2pt); 
\fill[opacity=0.3] (9.5,-0.1) circle (2pt); 
\fill[opacity=0.3] (10.5,-0.1) circle (2pt); 
\fill[opacity=0.3] (9.5,0.3) circle (2pt); 
\fill[opacity=0.3] (10.5,0.3) circle (2pt); 

\draw[very thick, opacity=0.3] (9,0.5) -- (10,0.5); 
\draw[very thick, opacity=0.3] (9,0.5) -- (9.5,0.3); 
\draw[very thick, opacity=0.3] (10,0.5) -- (9.5,0.3); 
\draw[very thick, opacity=0.3] (9.5,1.2) -- (9.5,0.3); 
\draw[very thick, opacity=0.3] (9.5,1.2) -- (10,0.5); 
\draw[very thick, opacity=0.3] (9.5,1.2) -- (9,0.5); 
\draw[very thick, opacity=0.3] (9.5,-0.1) -- (9.5,0.3); 
\draw[very thick, opacity=0.3] (9.5,-0.1) -- (10,0.5); 
\draw[very thick, opacity=0.3] (9.5,-0.1) -- (9,0.5); 

\draw[very thick, opacity=0.3] (11,0.5) -- (10.5,0.3); 
\draw[very thick, opacity=0.3] (11,0.5) -- (10.5,-0.1);
\draw[very thick, opacity=0.3] (11,0.5) -- (10.5,1.2);
\draw[very thick, opacity=0.3] (11,0.5) -- (10,0.5);

\draw[very thick, opacity=0.3] (10,0.5) -- (10.5,0.3); 
\draw[very thick, opacity=0.3] (10,0.5) -- (10.5,-0.1);
\draw[very thick, opacity=0.3] (10,0.5) -- (10.5,1.2);

\draw[very thick, opacity=0.3] (10.5,1.2) -- (10.5,0.3); 
\draw[very thick, opacity=0.3] (10.5,-0.1) -- (10.5,0.3); 

\draw[very thick, opacity=0.3] (9.5,0.3) -- (10.5,0.3); 
\draw[very thick, opacity=0.3] (9.5,1.2) -- (10.5,1.2);
\draw[very thick, opacity=0.3] (9.5,-0.1) -- (10.5,-0.1);  
\end{tikzpicture}
}
\end{tabular}
\caption{Contact graphs with $c(n,3)$ contacts, for $n=1, 2, 3, 4, 5$ (trivial cases) and largest known number of contacts, for $n=6,7, 8, 9$. For $n=1, 2, 3, 4, 5$ the maximal contact graphs are unique up to isometry. All the packings listed are minimally rigid and only for $n=9$, the packing is not rigid as the two bipyramids can be twisted slightly about the common pivot (see Section \ref{sec:empirical}).}
\end{figure}

Theorem~\ref{Harborth} implies in a straightforward way that 

\begin{equation}\label{asymp-planar}
\lim_{n\to +\infty}\frac{3n-c(n,2)}{\sqrt{n}}=\sqrt{12}=3.464\dots \ . 
\end{equation}

Although one cannot hope for an explicit formula for $c(n,3)$ in terms of $n$, there might be a way to prove a proper analogue of (\ref{asymp-planar}) in $\mathbb{E}^{3}$. 

To this end we know only what is stated in Theorem \ref{Bezdek-etc}. In order to state these results we need an additional concept. Let us imagine that we generate packings of $n$ unit balls in $\mathbb{E}^{3}$ in such a special way that each and every center of the $n$ unit balls chosen is a lattice point of the face-centered cubic lattice with shortest non-zero lattice vector of length $2$. Then let $c_{\rm fcc}(n)$ denote the largest possible contact number of all packings of $n$ unit balls obtained in this way.

The motivation for considering $c_{\rm fcc}(n)$ is obvious. Since in the planar case, the densest disk packing arrangement also maximizes contacts between disks and the face-centered cubic lattice is the densest for sphere packings in ${\mathbb{E}}^{3}$ \cite{Hales}, it makes sense to consider $c_{\rm fcc}(n)$ as a candidate for $c(n,3)$. Moreover, it is easy to see that $c_{\rm fcc}(2) = c(2, 3) = 1$, $c_{\rm fcc}(3) = c(3, 3) = 3$ and $c_{\rm fcc}(4) = c(4, 3) = 6$.  

\begin{theorem}\label{Bezdek-etc}

\item(i) \hskip0.6cm $c(n,3) < 6n-0.926n^{\frac{2}{3}}$, for all $n\ge 2$.

\item(ii) \hskip0.5cm $c_{\rm fcc}(n) < 6n-\frac{3\sqrt[3]{18\pi}}{\pi}n^{\frac{2}{3}}=6n-3.665\dots n^{\frac{2}{3}}$, for all $n\ge 2$.

\item(iii) \hskip0.4cm $6n-\sqrt[3]{486}n^{\frac{2}{3}}<2k(2k^2-3k+1)\leq c_{\rm fcc}(n)\leq c(n,3)$, for all $n=\frac{k(2k^2+1)}{3}$ with $k\ge 2$.

\end{theorem}

Recall that $(i)$ was proved in \cite{BeRe} (using the method of \cite{Be12}), while $(ii)$ and $(iii)$ were proved in \cite{Be12}. Clearly, Theorem~\ref{Bezdek-etc} implies that 
\begin{equation}\label{Be-Re}
0.926<\frac{6n-c(n,3)}{n^{\frac{2}{3}}}< \sqrt[3]{486}=7.862\dots ,
\end{equation} 
for all $n=\frac{k(2k^2+1)}{3}$ with $k\ge 2$.

Now consider the complexity of recognizing contact graphs of congruent sphere packings in ${\mathbb{E}}^{3}$. Just like its 2-dimensional analogue, Hlin\v{e}n\'{y} \cite{Hl} showed the 3-dimensional problem to be NP-hard by reduction from 3-SAT. In fact, the same is true in four dimensions \cite{Hl}.

\begin{theorem}\label{3-hard}
The problem of recognizing contact graphs of unit sphere packings in ${\mathbb{E}}^3$ (resp., ${\mathbb{E}}^4$) is NP-hard.
\end{theorem}

\section{Empirical approaches}\label{sec:empirical}
\begin{table}[ht]\label{table:values}
\centering
\footnotesize\setlength{\tabcolsep}{2.5pt}
\begin{tabular}{ccccc}
\hline\noalign{\smallskip}
$n$ & Lower bound \cite{Be12} & fcc upper bound \cite{Be12} & General upper bound \cite{BeRe} & (Putatively) Largest\\
 & $2k(2k^2-3k+1)$   & $\left\lfloor 6n-\frac{3\sqrt[3]{18\pi}}{\pi}n^{2/3}\right\rfloor$ & $\left\lfloor 6n-0.926n^{2/3}\right\rfloor$ & \cite{ArMaBr}, \cite{HC}, \cite{Hoy} \\
\noalign{\smallskip}\hline\noalign{\smallskip}
2 & & 6 & 10 & $1\ (= 3n - 5)$(trivial)\\
3 & & 10 & 16 & $3\ (= 3n - 6)$(trivial)\\
4 & & 14 & 21 & $6\ (= 3n - 6)$(trivial)\\
5 & & 19 & 27 & $9\ (= 3n - 6)$(trivial)\\ 
6 & 12 & 23 & 32 & $12^{*}\ (= 3n - 6)$\\
7 & & 28 & 38 & $15^{*}\ (= 3n - 6)$\\
8 & & 33 & 44 & $18^{*}\ (= 3n - 6)$\\
9 & & 38 & 49 & $21^{*}\ (= 3n - 6)$\\
10 & & 42 & 55 & $25^{*}\ (= 3n - 5)$\\ 
11 & & 47 & 61 & $29^{*}\ (= 3n - 4)$\\ 
12 & & 52 & 67 & $33^{**}\ (= 3n - 3)$\\
13 & & 57 & 72 & $36^{**}\ (= 3n - 3)$\\
14 & & 62 & 78 & $40^{**}\ (= 3n - 2)$\\
15 & & 67 & 84 & $44^{**}\ (= 3n - 1)$\\
16 & & 72 & 90 & $48^{**}\ (= 3n)$\\
17 & & 77 & 95 & $52^{**}\ (= 3n + 1)$\\
18 & & 82 & 101 & $56^{**}\ (= 3n + 2)$\\
19 & 60 & 87 & 107 & $60^{**}\ (= 3n + 3)$\\
\noalign{\smallskip}\hline
\end{tabular}
\normalsize
\caption{Bounds on the contact numbers of sphere packings in 3-space. The second column lists the lower bound when $n$ equals an octahedral number, i.e., $n=\frac{k(2k^2+1)}{3}$, for some $k=2,3,\ldots$. The third column lists the upper bound for packings on the face-centered cubic (fcc) lattice for all $n$, while the fourth column contains the general upper bound for all $n$. The final column contains the trivially known exact values for $n=2,3,4,5$ and the largest contact numbers found by the empirical approaches (for $n=6,7,8,9,10$ from \cite{ArMaBr}, for $n=11$ from \cite{Hoy} and for $n=12,\ldots,19$ from \cite{HC}). An asterisk * in the last column indicates the largest known contact number for minimally rigid clusters, while a double asterisk ** indicates the largest known contact number for rigid clusters.}
\end{table}

Throughout this section, we deal with finite unit sphere packings in three dimensional Euclidean space, that is, with finite packings of unit balls in $\mathbb{E}^{3}$. Therefore, in this section a `sphere' always means a unit sphere in $\mathbb{E}^{3}$. Taking inspiration from materials science and statistical physics, we will often refer to a finite sphere packing as a \textit{cluster}. Our aim is to describe three computational approaches that have recently been employed in constructing putatively maximal contact graphs for packings of $n$ spheres under certain rigidity assumptions. 

\begin{definition}[Minimal rigidity \cite{ArMaBr}]
A cluster of $n\geq 4$ unit spheres is said to be \textit {minimally rigid} if 
\begin{itemize}
\item each sphere is in contact with at least 3 others, and 
\item  the cluster has at least $3n-6$ contacts (that is, the corresponding contact graph has at least $3n-6$ edges).  
\end{itemize}
\end{definition}

\begin{definition}[Rigidity \cite{HC}]\label{rigidity}
A cluster of $n$ unit spheres is (nonlinearly) rigid if it cannot be deformed continuously by any finite amount and still maintain all contacts \cite{HC}.
\end{definition}

The first two approaches - which we discuss together - deal with minimally rigid clusters, while the third investigates rigid clusters. We observe that one can find minimally rigid clusters that are not rigid. The paper \cite{ArMaBr} contains such an example for $n=9$ (Fig. 1 (i)).

\subsection{Contact number estimates for up to 11 spheres}

Arkus, Manoharan and Brenner \cite{ArMaBr} made an attempt to exhaustively generate all minimally rigid packings of $n$ spheres that are either local or global maxima of the number of contacts. Here a packing is considered a global maximum if the spheres in the cluster cannot form any additional contacts or a local maximum if new contacts can only be created after breaking an existing contact. They produce a list of maximal contact minimally rigid sphere packings for $n=2,\ldots, 9$, which is putatively complete up to possible omissions due to round off errors, and a partial list for 10 spheres. Since the number of such packings grows exponentially with $n$, their approach can only be implemented on a computer.   

Before we delve into the details of their methodology, it would be pertinent to understand why it focuses on finding minimally rigid clusters. It seems the minimal rigidity was considered due to two reasons: First, is Maxwell's criterion \cite{Max}, which is popular in physics literature and states that a rigid cluster of $n$ spheres has at least $3n-6$ contacts. This is false as in \cite{HC} examples of rigid clusters with $n\geq 10$ have been reported that are not minimally rigid. Second, is the intuition that any maximum contact cluster of $n\geq 4$ spheres should be minimally rigid. Up to our knowledge, there exists no proof of or counterexample to this intuition. We can, however, prove the following. The proof depends on the assumption that Arkus et al. \cite{ArMaBr} have found all minimally rigid packings of $n\leq 9$ spheres that maximize the number of contacts.

\begin{proposition}\label{empirical}
Assume that all maximal contact minimally rigid packings of $n\leq 9$ spheres are listed in \cite{ArMaBrArxiv} and \cite{ArMaBr}, then for $n=4,\ldots, 9$, 
\[c(n,3)=3n-6, 
\] 
and there exists a minimally rigid cluster with $c(n,3)$ contacts. 
\end{proposition}

\begin{proof}
We introduce some terminology for finite sphere packings and their contact graphs. We say that any three pairwise touching spheres form a \textit{triangle}. A triangle is called an \textit{exposed triangle} if an additional sphere, not part of the original packing, can be brought in contact with all the three spheres in the triangle without overlapping with any sphere already in the packing. Triangles and exposed triangles can be equivalently defined in terms of contact graphs. 

\vspace{2mm}

\noindent \textit{Claim}: Any maximal contact graph on $n$ vertices with $4\leq n\leq 9$ has an exposed triangle and each vertex of such a graph has degree at least 3. 

\vspace{2mm}

By checking the list of all minimally rigid packings of $4\leq n\leq 9$ sphere given in \cite{ArMaBrArxiv} \footnote{The complete list (up to possible omissions due to round off errors) of minimally rigid packings of $n\leq 9$ spheres and a preliminary list of $n=10$ spheres appears on the arXiv \cite{ArMaBrArxiv}. The paper \cite{ArMaBr} only contains a partial list, so for the more complete list we refer to the arXiv version.} exhaustively, we see that the claim holds for all such sphere packings. We now proceed by induction on $n$. 

For $n=4$, there is only one maximal contact graph and for that Claim holds. Now suppose Claim holds for some $n\geq 4$. Consider any contact graph $G$ that has the largest number of contacts among all contact graphs having $n+1$ vertices. Let $v$ be a vertex of $G$.  

Suppose that $v$ has degree 2. Then $G-\{v\}$, the graph obtained by deleting $v$ and all edges incident to $v$ from $G$, must be a maximal contact graph on $n$ vertices, since if this is not the case, then replacing $G-\{v\}$ by a maximal contact graph $H$ on $n$ vertices and joining $v$ to any exposed triangle of $H$ produces a contact graph on $n+1$ vertices with strictly more contacts than $G$. But then $G-\{v\}$ has an exposed triangle and joining $v$ to that triangle produces a contact graph on $n+1$ vertices with strictly more contacts than $G$. This is a contradiction and so $v$ has degree at least 3. Thus $G$ is minimally rigid and has an exposed triangle. This completes the proof of Claim. 

Thus for $n=4,\ldots, 9$, the list of all maximal contact graphs coincides with the list of minimally rigid maximal contact graphs that have $3n-6$ contacts according to \cite{ArMaBr}. 
\end{proof}

We now describe the approach of Arkus et al. \cite{ArMaBr}. Note that since we are dealing with unit spheres (as stated in the opening of this section), the distance between the centers of two touching spheres is 2. Let $n\geq 4$ be a positive integer.

\vspace{4mm}

\noindent \textbf{Procedure 1 (\cite{ArMaBr}):}

\vspace{2mm}

\noindent \textit{Step1}: List the adjacency matrices of all nonisomorphic simple graphs with $n$ vertices and exactly $3n-6$ edges such that each vertex has degree at least 3. Let $\cal{A}$ be the set of all such adjacency matrices. In \cite{ArMaBr}, this step is performed using the graph isomorphism testing program \textit{nauty} and Sage package \textit{nice}. 

\vspace{2mm}

\noindent \textit{Step 2}: For each $A\in \cal{A}$, there is a corresponding simple graph $G_{A}$ with vertex set (say) $V=\{v_{1},\ldots, v_{n}\}$. Denote the $(i,j)$-entry of $A$ by $A_{ij}$ and consider each vertex $v_{i}$ of $G_{A}$ as a point $v_{i}=(x_{i},y_{i},z_{i})$ (the coordinates are yet unknown) in ${\mathbb{E}}^{3}$. Then $G_{A}$ is a contact graph if and only if we can place congruent spheres centered at the vertices of $G_{A}$ such that none of the spheres overlap and $A_{ij}=1$ implies that the spheres centered at $v_{i}$ and $v_{j}$ touch. Use the simple geometric elimination rules derived in \cite{ArMaBr} to remove a substantial number of adjacency matrices from $\cal{A}$ that cannot be realized into contact graphs. These geometric rules basically detect certain patterns that cannot occur in the adjacency matrices of contact graphs. Let the resulting set of adjacency matrices be denoted by $\cal{B}$.  

\vspace{2mm}

\noindent \textit{Step 3}: For any $A\in \cal{B}$, $G_{A}$ is a contact graph of a unit sphere packing if and only if for $i> j$ the system of nonlinear equations 
\begin{equation}\label{eq:arkus}
\begin{aligned}
D_{ij}^{2} = (x_{i}-x_{j})^{2}+(y_{i}-y_{j})^{2}+(z_{i}-z_{j})^{2} = 2, \ \ \ \  A_{ij} = 1,\\
D_{ij}^{2} = (x_{i}-x_{j})^{2}+(y_{i}-y_{j})^{2}+(z_{i}-z_{j})^{2} \geq 2, \ \ \ \  A_{ij} = 0.
\end{aligned}
\end{equation}
has a real solution. Note that we are only considering $i>j$ as the matrix $A$ is symmetric. Without loss of generality, we can assume that $x_{1}= y_{1}= z_{1}=0$ (the first sphere is centered at the origin); $y_{2}=z_{2}=0$ (the second sphere lies on the $x$-axis) and $z_{3}=0$ (the third sphere lies in the $xy$-plane). Therefore, we obtain a system with $\frac{n(n-1)}{2}$ constraints (of which $3n-6$ are equality constraints) in $3n-6$ unknowns. Here $D_{ij}$ is the distance between vertices $v_{i}$ and $v_{j}$. In \cite{ArMaBr}, for each $A\in\cal{B}$, the system \eqref{eq:arkus} is solved analytically for $n\leq 9$ and numerically for $n=10$. 

\vspace{2mm}

\noindent \textit{Step 4}: Form the distance matrix $D_{A}=[D_{ij}]$ and let $\cal{D}$ be the set of all distance matrices corresponding to valid contact graphs of packings of $n$ unit spheres. The contact number corresponding to any $D\in\cal{D}$ equals the number of entries of $D$ that equal 2 and lie above (equivalently below) the main diagonal. Note that, although we started with the adjacency matrices corresponding to exactly $3n-6$ contacts, solving system \eqref{eq:arkus} yields all distance matrices with $3n-6$ or more contacts. 

\vspace{4mm}

Since the geometric rules used in Step 2, are susceptible to round off errors, there is a possibility that some adjacency matrices are incorrectly eliminated from $\cal{A}$. Also for $n=10$, Newton's method was used to solve \eqref{eq:arkus} as the computational limit of analytical methods was reached for packings of 10 spheres. Thus the list of minimally rigid sphere packings provided in \cite{ArMaBr} could potentially be incomplete. As a result, the contact number $c(n,3)$ is still unknown for $n\geq 6$.\footnote{According to \cite{ArMaBr}, for $n\leq 7$, it is possible to solve the system \eqref{eq:arkus} using standard algebraic geometry methods for all $A\in\cal{A}$ without filtering by geometric rules. Arkus et al. \cite{ArMaBr} attempted this using the package \textit{SINGULAR}. Therefore, most likely for $n=6,7$, the maximal contact graphs as obtained in \cite{ArMaBr} are optimal for minimally rigid sphere packings.} Nevertheless, it is quite reasonable to conjecture the following.  

\begin{conjecture}
For $n\geq 6$, every contact graph of a packing of $n$ spheres with $c(n, 3)$ contacts is minimally rigid. Moreover, for $n=6,\ldots, 9$,
\[c(n, 3) = 3n-6.
\] 
\end{conjecture}

Hoy et al. \cite{Hoy} extended Procedure 1 to 11 spheres. However, they employ Newton's method, which cannot guarantee to obtain a solution, whenever one exists. Also they make the erroneous assumption that the contact graph of any minimally rigid sphere packing contains a Hamiltonian path. This assumption greatly reduces the number of adjacency matrices to be considered. However, Connelly, E. Demaine and M. Demaine showed this to be false \cite{HayBlog}, providing a counterexample with 16 vertices.

\subsection{Maximal contact rigid clusters}

Despite its intuitive significance, we have seen that minimal rigidity is neither sufficient nor necessary for rigidity. Holmes-Cerfon \cite{HC} developed a computational technique to potentially construct all rigid sphere packings of a small number of spheres. Her idea to consider rigid clusters comes from the intuition that in a physical system (like self-assembling colloids), a rigid cluster is more likely to form and survive than a non-rigid cluster. 

Some aspects of Holmes-Cerfon's method are similar to the empirical approaches described earlier. For instance, the mathematical formulation in terms of adjacency and distance matrices, and use of system \eqref{eq:arkus} to arrive at potential solutions remains unchanged. However, there are two fundamental differences. 

Obviously, one is the consideration of rigidity instead of minimal rigidity. This results in the removal of the restriction that the contact graph should have $3n-6$ edges. Instead, any solution obtained is tested for rigidity. 

The second major difference lies in the way all potential solutions are reached. In the previous approaches, the method involved an exhaustive adjacency matrix search followed by filtering through some geometrical rules. Here the procedure starts with a single rigid packing $\cal{P}$ of $n$ spheres and attempts to generate all other rigid packings of $n$ spheres as follows: Break an existing contact in $\cal{P}$ by deleting an equation from \eqref{eq:arkus}. This usually leads to a single internal degree of freedom that results in a one-dimensional solution set. When this happens, one can follow the one-dimensional path numerically until another contact is formed, typically resulting in another rigid packing \cite{HC}. 

The paper \cite{HC} provides a preliminary list for all rigid sphere packings of up to 14 spheres and maximal contact packings of up to 19 spheres. However, the use of numerical methods and approximations throughout means that the list is potentially incomplete.

\section{Digital and totally separable sphere packings for $d=2, 3$}\label{sec:polyomino}
\begin{figure}[ht] 
\centering
\begin{tikzpicture}[scale=0.75]
\draw[fill=black!30] (1.4,4) -- (3.4,4) -- (4.8,5) -- (2.8,5) -- cycle;
\draw[fill=black!30] (3.4,1) -- (3.4,4) -- (4.8,5) -- (4.8,2) -- cycle;
\draw[fill=black!30] (3.4,1) -- (3.4,4) -- (4.8,5) -- (4.8,2) -- cycle;
\draw[fill=black!60] (1.4,1) -- (3.4,1) -- (3.4,4) -- (1.4,4) -- cycle;
\draw[fill=black!60] (0.7,0.5) -- (1.7,0.5) -- (1.7,1.5) -- (0.7,1.5) -- cycle;
\draw[fill=black!30] (1.7,0.5) -- (1.7,1.5) -- (2.4,2) -- (2.4, 1) -- cycle;
\draw[fill=black!30] (0.7,1.5) -- (1.7,1.5) -- (2.4,2) -- (1.4,2) -- cycle;

\draw (1.4,2) -- (3.4,2);
\draw (1.4,3) -- (3.4,3);
\draw (3.4,2) -- (4.8,3);
\draw (3.4,3) -- (4.8,4);
\draw (2.4,2) -- (2.4,4);
\draw (2.4,4) -- (3.8,5);
\draw (4.1,1.5) -- (4.1,4.5);
\draw (4.1,4.5) -- (2.1,4.5);

\draw[fill=black!30] (8.4,4) -- (10.4,4) -- (11.8,5) -- (9.8,5) -- cycle;
\draw[fill=black!30] (10.4,1) -- (10.4,4) -- (11.8,5) -- (11.8,2) -- cycle;
\draw[fill=black!30] (10.4,1) -- (10.4,4) -- (11.8,5) -- (11.8,2) -- cycle;
\draw[fill=black!60] (8.4,1) -- (10.4,1) -- (10.4,4) -- (8.4,4) -- cycle;
\draw[fill=black!60] (7.7,0.5) -- (8.7,0.5) -- (8.7,1.5) -- (7.7,1.5) -- cycle;
\draw[fill=black!30] (8.7,0.5) -- (8.7,1.5) -- (9.4,2) -- (9.4, 1) -- cycle;
\draw[fill=black!30] (7.7,1.5) -- (8.7,1.5) -- (9.4,2) -- (8.4,2) -- cycle;

\draw (8.4,2) -- (10.4,2);
\draw (8.4,3) -- (10.4,3);
\draw (10.4,2) -- (11.8,3);
\draw (10.4,3) -- (11.8,4);
\draw (9.4,2) -- (9.4,4);
\draw (9.4,4) -- (10.8,5);
\draw (11.1,1.5) -- (11.1,4.5);
\draw (11.1,4.5) -- (9.1,4.5);

\draw (0,0) rectangle (5,5);
\foreach \i in {0,...,4} 
\foreach \j in {0,...,4}
{
\draw (\i,\j) rectangle (\i+1,\j+1);
}

\draw (0.7,0.5) rectangle (5.7,5.5);
\foreach \i in {0.7,...,4.7} 
\foreach \j in {0.5,...,4.5}
{
\draw (\i,\j) rectangle (\i+1,\j+1);
}

\draw (1.4,1) rectangle (6.4,6);
\foreach \i in {1.4,...,5.4} 
\foreach \j in {1,...,5}
{
\draw (\i,\j) rectangle (\i+1,\j+1);
}

\draw (2.1,1.5) rectangle (7.1,6.5);
\foreach \i in {2.1,...,6.1} 
\foreach \j in {1.5,...,5.5}
{
\draw (\i,\j) rectangle (\i+1,\j+1);
}

\draw (2.8,2) rectangle (7.8,7);
\foreach \i in {2.8,...,6.8} 
\foreach \j in {2,...,6}
{
\draw (\i,\j) rectangle (\i+1,\j+1);
}

\foreach \i in {0,...,5} 
\foreach \j in {0,...,5}
{
\draw (\i,\j) -- (\i+2.8,\j+2);
}
\shade[ball color=white!70!black] (13.6,4.4) circle (.6cm);
\shade[ball color=white!70!black] (14.8,4.4) circle (.6cm);

\shade[ball color=white!70!black] (13.6,3.2) circle (.6cm);
\shade[ball color=white!70!black] (14.8,3.2) circle (.6cm);

\shade[ball color=white!70!black] (13.6,2) circle (.6cm);
\shade[ball color=white!70!black] (14.8,2) circle (.6cm);

\shade[ball color=white!70!black] (13,3.9) circle (.6cm);
\shade[ball color=white!70!black] (14.2,3.9) circle (.6cm);

\shade[ball color=white!70!black] (13,2.7) circle (.6cm);
\shade[ball color=white!70!black] (14.2,2.7) circle (.6cm);

\shade[ball color=white!70!black] (13,1.5) circle (.6cm);
\shade[ball color=white!70!black] (14.2,1.5) circle (.6cm);

\shade[ball color=white!70!black] (12.4,1) circle (.6cm);

\end{tikzpicture}
\caption{A polyomino of volume 13 and the corresponding digital packing of 13 spheres.}\label{fig:polyominoes}
\end{figure}
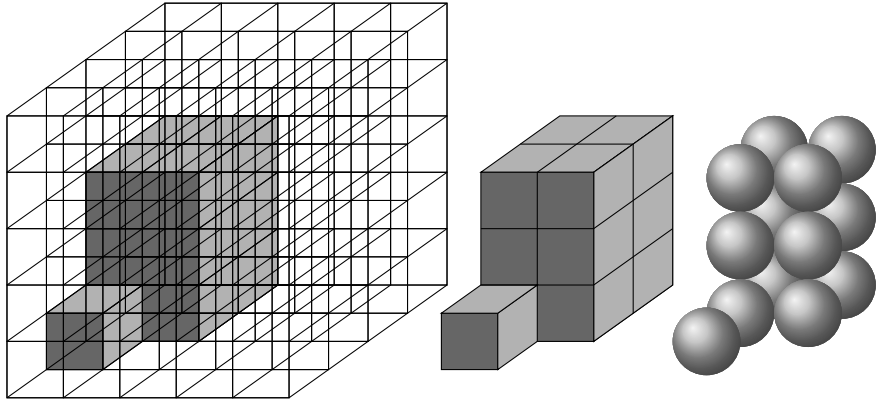
In this section, we use the terms `cube', `sphere' and `ball' to refer to two and three dimensional objects of these types. Consider the 3-dimensional (resp. 2-dimensional) integer lattice ${\mathbb{Z}}^{3}$ (resp. ${\mathbb{Z}}^{2}$), which can be thought of as an infinite space tiling array of unit cubes called \textit{lattice cells}. For convenience, we imagine these cubes to be centered at the integer points, rather than having their vertices at these points. Two lattice cells are \textit{connected} if they share a facet. 

We refer to a packing of congruent unit diameter spheres centered at the points of ${\mathbb{Z}}^{3}$ (resp. ${\mathbb{Z}}^{2}$) as a \textit{digital sphere packing}. These packings provide a natural means for generating totally separable sphere packings. We denote the maximal contact number of such a digital packing of $n$ spheres by $c_{{\mathbb{Z}}}(n, 3)$ (resp. $c_{{\mathbb{Z}}}(n,2)$). Clearly, $c_{{\mathbb{Z}}}(n,2)\leq c_{\rm sep}(n, 2)$ and $c_{{\mathbb{Z}}}(n,3)\leq c_{\rm sep}(n, 3)$. The question is how large the maximum digital contact number can be and whether it equals the corresponding maximum contact number of totally separable sphere packings. 

A 3-dimensional (resp. 2-dimensional) \textit{polyomino} is a finite collection of connected lattice cells of ${\mathbb{Z}}^{3}$ (resp. ${\mathbb{Z}}^{2}$). Considering the maximum volume ball contained in a cube, each polyomino corresponds to a digital sphere (circle) packing and vice versa. Moreover, since the ball (circle) intersects the cube (square) at 6 points (4 points), one on each facet, it follows that the number of facets shared between the cells of the polyomino equals the contact number of the corresponding digital packing. Figure \ref{fig:polyominoes} shows a portion of the cubic lattice centered at the points of ${\mathbb{Z}}^{3}$, a 3-dimensional polyomino and the digital sphere packing corresponding to that polyomino.

It is easy to see that minimizing the surface area (resp., perimeter) of a 3-dimensional (resp., 2-dimensional) polyomino of volume $n$ corresponds to finding the maximum contact number of a digital packing of $n$ spheres. Harary and Harborth \cite{HaHa} studied the problem of finding isoperimetric polyominoes of area $n$ in 2-space. Their key insight was that $n$ squares can be arranged in a square-like arrangement so as to minimize the perimeter of the resulting polyomino. The same construction appears in \cite{AlCe}, but without referencing \cite{HaHa}.  The 3-dimensional case has a similar solution which first appeared in \cite{AlCe}. The proposed arrangement consists of forming a quasi-cube (an orthogonal box with one or two edges deficient by at the most one unit) followed by attaching as many of the remaining cells as possible in the form of a quasi-square layer. The rest of the cells are then attached to the quasi-cube in the form of a row. The main results of \cite{HaHa} and \cite{AlCe} on isoperimetric polyominoes in ${\mathbb{E}}^{2}$ and ${\mathbb{E}}^{3}$ can be used to derive the following about the maximum digital contact numbers (see \cite{BeSzSz}).   

\begin{theorem}\label{polyomino}
Given $n\ge 2$, we have 

\item(i) \hskip0.6cm $c_{{\mathbb{Z}}}(n,2)=\lfloor 2n- 2\sqrt{n}\rfloor.$ 

\item(ii) \hskip0.5cm $c_{{\mathbb{Z}}}(n,3)= 3n- 3n^{\frac{2}{3}} - o(n^{\frac{2}{3}}).$ 

\end{theorem}

We now turn to the more general totally separable sphere packings in ${\mathbb{E}}^{2}$ and ${\mathbb{E}}^{3}$. The contact number problem for such packings was discussed in the very recent paper \cite{BeSzSz}.

\begin{theorem}\label{BeSzSz}
For all $n\ge 2$, we have 

\item(i) \hskip0.6cm $c_{\rm sep}(n,2)=\lfloor 2n- 2\sqrt{n}\rfloor$. 

\item(ii) \hskip0.5cm $3n- 3n^{\frac{2}{3}} - o(n^{\frac{2}{3}}) \leq c_{\rm sep}(n,3)<3n-1.346n^{\frac{2}{3}}$. 
\end{theorem} 

Part (i) follows from a natural modification of Harborth's proof \cite{Ha} of Theorem \ref{Harborth} (for details see \cite{BeSzSz}). The lower bound in (ii) comes from the fact that every digital sphere packing is totally separable. However, proving the upper bound in (ii) is more involved. 

Theorem \ref{BeSzSz} can be used to generate the following analogues of relations \eqref{asymp-planar} and \eqref{Be-Re}. 

\begin{equation}\label{eq:asymp-planar-sep}
\lim_{n\to +\infty}\frac{2n-c_{\rm sep} (n,2)}{\sqrt{n}}=2\ . 
\end{equation}

\begin{equation}\label{eq:Be-Re-sep}
1.346< \frac{3n-c_{\rm sep}(n,3)}{n^{\frac{2}{3}}}\leq 3 + o(1)\ .
\end{equation} 

Since the bounds in \eqref{eq:Be-Re-sep} are tighter than \eqref{Be-Re}, it is reasonable to conjecture that the limit of $\frac{3n-c_{\rm sep}(n,3)}{n^{2/3}}$ exists as $n\to +\infty$. In fact, it can be asked if this limit equals 3. Furthermore, a comparison of Theorem \ref{polyomino} and Theorem \ref{BeSzSz} shows that $c_{\rm sep}(n,2) = c_{{\mathbb{Z}}}(n,2)$ holds for all positive integers $n$. Therefore, it is natural to raise the following open problem. 

\begin{problem}\label{sep=digital}
Show that 
\[\lim_{n\to +\infty} \frac{3n-c_{\rm sep}(n,3)}{n^{\frac{2}{3}}} = 3\ .
 \]
Moreover, is it the case that $c_{\rm sep}(n,3) = c_{{\mathbb{Z}}}(n,3)$, for all positive integers $n$? If not, then characterize those values of $n$ for which this holds. 
\end{problem}

\section{On largest contact numbers in higher dimensional spaces}\label{sec:d-dim}

In this section, we study the contact number problem in ${\mathbb{E}}^{d}$, both for packings of ${\mathbf B}^{d}$ and translates of an arbitrary $d$-dimensional convex body ${\mathbf K}$.    

\subsection{Packings by translates of a convex body}
One of the main results of this section is an upper bound for the number of touching pairs in an arbitrary finite packing of translates of a convex body, proved in \cite{B02}. 
In order to state the theorem in question in a concise way we need a bit of notation. Let $\mathbf{K}$ be an arbitrary convex body in $\mathbb{E}^{d}$, $d\geq 3$. Then let $\delta(\mathbf{K})$ denote the density of a densest packing of translates of the convex body $\mathbf{K}$ in $\mathbb{E}^{d}$, $d\geq 3$. Moreover, let $$\text{iq}(\mathbf{K}):=\frac{\left(\text{svol}_{d-1}(\text{bd}\mathbf{K})\right)^d}{\left(\text{vol}_d(\mathbf{K})\right)^{d-1}}$$
be the isoperimetric quotient of the convex body $\mathbf{K}$, where $\text{svol}_{d-1}(\text{bd}\mathbf{K})$ denotes the $(d-1)$-dimensional surface volume of the boundary $\text{bd}\mathbf{K}$ of $\mathbf{K}$ and $\text{vol}_d(\mathbf{K})$ denotes the $d$-dimensional volume of
$\mathbf{K}$. Furthermore, let $H(\mathbf{K})$ denote the Hadwiger number of $\mathbf{K}$, which is the largest number of non-overlapping translates of $\mathbf{K}$ that can all touch $\mathbf{K}$.
An elegant observation of Hadwiger \cite{H} is that $H(\mathbf{K})\leq 3^d-1$, where equality holds if and only if  $\mathbf{K}$ is an affine $d$-cube.
Finally, let the one-sided Hadwiger number $h(\mathbf{K} )$ of $\mathbf{K}$ be the largest number of non-overlapping translates of $\mathbf{K}$ that touch $\mathbf{K}$ and that all lie in a closed supporting halfspace of $\mathbf{K}$. In \cite{BB}, using the Brunn--Minkowski inequality, it is proved that $h(\mathbf{K} )\leq 2\cdot 3^{d-1}-1$, where equality is attained if and only if $\mathbf{K}$ is an affine $d$-cube.
Let $\mathbf{K_o}:=\frac{1}{2}(\mathbf{K}+(-\mathbf{K}))$ be the normalized (centrally symmetric) difference body assigned to $\mathbf{K}$.

\begin{theorem}\label{15}
Let $\mathbf{K}$ be an arbitrary convex body in $\mathbb{E}^{d}$, $d\geq 3$. Then
\begin{align*}
c({\bf K},n,d) &\leq  \frac{H(\mathbf{K_o})}{2}\, n-\frac{1}{2^{d}\delta(\mathbf{K_o})^{\frac{d-1}{d}}}\sqrt[d]{\frac{\mathrm{iq}(\mathbf{B}^d)}{\mathrm{iq}(\mathbf{K_o})}}\; n^{\frac{d-1}{d}}-(H(\mathbf{K_o})-h(\mathbf{K_o})-1) \\
&\leq \frac{3^{d}-1}{2}\, n-\frac{\sqrt[d]{\omega_{d}}}{2^{d+1}}\: n^{\frac{d-1}{d}},  
\end{align*}
where $\omega_d=\frac{\pi^{\frac{d}{2}}}{\Gamma (\frac{d}{2}+1) }={\rm vol}_d(\mathbf{B}^d)$.
\end{theorem}

Since for the most part we are interested in contact numbers of sphere packings, it would be interesting to see the form Theorem~\ref{15} takes when $\mathbf{K}=\mathbf{B}^{d}$. Recall that $k(d)$ denotes the kissing number of a unit ball in ${\mathbb{E}}^{d}$. Let $\delta_{d}$ stand for the largest possible density for (infinite) packings of unit balls in ${\mathbb{E}}^{d}$. The following consequence of Theorem~\ref{15} was reported in \cite{Be12}. 
 
\begin{corollary}\label{d-sphere}
Let $n > 1$ and $d \geq 3$ be positive integers. Then 
\[c(n,d) < \frac{1}{2}k(d)\ n - \frac{1}{2^d} \delta_{d}^{-\frac{d-1}{d}}n^{\frac{d-1}{d}}. 
\]
\end{corollary}

Now, recall the well-known theorem of Kabatiansky and Levenshtein \cite{KaLe} that $k(d) \leq 2^{0.401d(1+o(1))}$ and $\delta_{d} \leq 2^{-0.599d(1+o(1))}$ as $d \to +\infty$. Together with Corollary \ref{d-sphere} this gives 
\[
￼￼c(n,d) < \frac{1}{2}2^{0.401d(1+o(1))} n - \frac{1}{2^d} 2^{0.599(1+o(1))(d-1)}n^{\frac{d-1}{d}}, 
\]
for $n>1$, as $d \to +\infty$. 

In particular, for $d=3$ we have $k(3) = 12$ \cite{ScWa} and $\delta_{3} = \frac{\pi}{\sqrt{18}}$ \cite{Hales}. Thus, by combining these with Corollary \ref{d-sphere} we find that for $n>1$,  
\[￼￼c(n,3) < 6n-\frac{1}{8}\left(\frac{\pi}{\sqrt{18}}\right)^{-\frac{2}{3}}n^{\frac{2}{3}} = 6n-0.152\ldots n^{\frac{2}{3}}.
\]

The above upper bound for $c(n,3)$ was substantially improved, first in \cite{Be12} and then further in \cite{BeRe}. The current best upper bound is stated in Theorem \ref{Bezdek-etc} (i). 

In the proof of Theorem~\ref{15} published in \cite{B02}, the following statement plays an important role that might be of independent interest and so we quote it as follows. For the sake of completeness we wish to point out that Theorem~\ref{Bezdek-15} and Corollary~\ref{Bezdek-15-Corollary}, are actual strengthenings of Theorem 3.1 and Corollary 3.1 of \cite{BHW94} mainly because, in our case the containers of the packings in question are highly non-convex.

\begin{theorem}\label{Bezdek-15}
Let $\mathbf{K_o}$ be a convex body in $\mathbb{E}^{d}$, $d\geq 2$ symmetric about the origin $\mathbf{o}$ of $\mathbb{E}^{d}$ and let $\{\mathbf{c}_1+\mathbf{K_o}, \mathbf{c}_2+\mathbf{K_o}, \dots , \mathbf{c}_n+\mathbf{K_o}\}$ be an arbitrary packing
of $n>1$ translates of $\mathbf{K_o}$ in $\mathbb{E}^{d}$. Then
$$ \frac{n{\rm vol}_d(\mathbf{K_o})}{{\rm vol}_d(\bigcup_{i=1}^n(\mathbf{c}_i+2\mathbf{K_o}))}\le \delta(\mathbf{K_o}).$$
\end{theorem}

The following is an immediate corollary of Theorem~\ref{Bezdek-15}.

\begin{corollary}\label{Bezdek-15-Corollary}
Let ${\cal P}_n(\mathbf{K_o})$ be the family of all possible packings of $n>1$ translates of the $\mathbf{o}$-symmetric convex body $\mathbf{K_o}$ in $\mathbb{E}^{d}$, $d\geq 2$. Moreover, let 
$$\delta (\mathbf{K_o}, n):=\max \bigg\{ \frac{n{\rm vol}_d(\mathbf{K_o})}{{\rm vol}_d(\bigcup_{i=1}^n(\mathbf{c}_i+2\mathbf{K_o}))}\ \bigg| \ \{ \mathbf{c}_1+\mathbf{K_o}, \dots , \mathbf{c}_n+\mathbf{K_o}\}\in {\cal P}_n(\mathbf{K_o}) \bigg\}.$$
Then
$$\limsup_{n\to\infty}\delta (\mathbf{K_o}, n)= \delta(\mathbf{K_o}) .$$

\end{corollary}

Interestingly enough one can interpret the contact number problem on the exact values of $c(n,d)$ as a volume minimization question. Here we give only an outline of that idea introduced and discussed in detail in \cite{BeLa15}.

\begin{definition}\label{defn:main}
Let $\P^n:=\{\c_i+\BB^d\ |\ 1\le i\le n \ {\rm with}\ \|\c_j-\c_k\|\ge 2\ {\rm for}\  {\rm all}\ 1\le j< k\le n\}$ be an arbitrary packing of $n>1$ unit balls in $\Eu^{d}$. The part of space covered by the unit balls of $\P^n$ is labelled by $\mathbf{P}^n:=\bigcup_{i=1}^{n}(\c_i+\BB^d)$. Moreover, let
$C^n:=\{\c_i\ |\ 1\le i\le n\}$ stand for the set of centers of the unit balls in $\P^n$. Furthermore, for any $\lambda >0$ let $\mathbf{P}^n_{\lambda}:=\bigcup \{ \mathbf{x}+\lambda\BB^d\ |\ \mathbf{x}\in \mathbf{P}^n\}=\bigcup_{i=1}^{n}(\c_i+(1+\lambda)\BB^d)$ denote the outer parallel domain of $\mathbf{P}^n$ having outer radius $\lambda$. Finally, let 

$$\delta_d(n, \lambda):=\max_{\P^n}\frac{n\omega_d}{\vol_d(\mathbf{P}^n_{\lambda})}=\frac{n\omega_d}{\min_{\P^n}  \vol_d\left(\bigcup_{i=1}^{n}(\c_i+(1+\lambda)\BB^d) \right)}$$ 
and
$$\delta_d(\lambda):=\limsup_{n\to +\infty}\delta_d(n, \lambda).$$
\end{definition}

Now, let $\P:=\{\c_i+\BB^d\ |\ i=1,2,\dots \ {\rm with}\ \|\c_j-\c_k\|\ge 2\ {\rm for}\  {\rm all}\ 1\le j< k\}$ be an arbitrary infinite packing of unit balls in $\Eu^{d}$. Recall that
the packing density $\delta_d$ of unit balls in $\Eu^{d}$ can be computed as follows:

$$\delta_d=\sup_{\P}\left(\limsup_{R\to +\infty}  \frac{\sum_{\c_i+\BB^d\subset R\BB^d}\vol_d(\c_i+\BB^d)}{\vol_d(R\BB^d)}\right).$$
Hence, it is rather easy to see that $\delta_d\le\delta_d(\lambda)$ holds for all $\lambda>0, d\ge 2$.
On the other hand, it was proved in \cite{B02} (see also Corollary~\ref{Bezdek-15-Corollary}) that {\it $\delta_d=\delta_d(\lambda)$ for all $\lambda\ge 1$} leading to the classical sphere packing problem. Now, we are ready to put forward the following question from \cite{BeLa15}.

\begin{problem}\label{core}
Determine (resp., estimate) $\delta_d(\lambda)$ for $d\ge 2$, $0<\lambda<\sqrt{\frac{2d}{d+1}}-1$.
\end{problem}

First, we note that $\frac{2}{\sqrt{3}}-1\le \sqrt{\frac{2d}{d+1}}-1$ holds for all $d\ge 2$. Second, observe that as $\frac{2}{\sqrt{3}}$ is the circumradius of a regular triangle of side length $2$, therefore if $0<\lambda<\frac{2}{\sqrt{3}}-1$, then for any unit ball packing $\P^n$ no three of the closed balls in the family $\{\c_i+(1+\lambda)\BB^d\ |\ 1\le i\le n\}$ have a point in common. In other words, for any $\lambda$ with $0<\lambda < \frac{2}{\sqrt{3}}-1$ and for any unit ball packing $\P^n$, in the arrangement $\{\c_i+(1+\lambda)\BB^d\ |\ 1\le i\le n\}$ of closed balls of radii $1+\lambda$ only pairs of balls may overlap. Thus, computing
$\delta_d(n, \lambda)$, i.e., minimizing $\vol_d(\mathbf{P}^n_{\lambda})$ means maximizing the total volume of pairwise overlaps in the ball arrangement $\{\c_i+(1+\lambda)\BB^d\ |\ 1\le i\le n\}$ with the underlying packing $\P^n$. Intuition would suggest to achieve this by simply maximizing the number of touching pairs in the unit ball packing $\P^n$. Hence, Problem~\ref{core} becomes very close to the {\it contact number problem} of finite unit ball packings for $0<\lambda<\frac{2}{\sqrt{3}}-1$. Indeed, we have the following statement proved in \cite{BeLa15}.

\begin{theorem}\label{contact numbers and isoperimetry}
{\it Let $n>1$ and $d>1$ be given. Then there exists $\lambda_{d, n}>0$ and a packing  $\widehat{\P}^n$ of $n$ unit balls in $\Eu^{d}$ possessing the largest contact number for the given $n$ such that for all $\lambda$ satisfying $0<\lambda< \lambda_{d, n}$, $\delta_d(n, \lambda)$ is generated by $\widehat{\P}^n$, i.e., $\vol_d(\mathbf{P}^n_{\lambda})\ge \vol_d(\widehat{\mathbf{P}}^n_{\lambda})$ holds for every packing $\P^n$ of $n$ unit balls in $\Eu^{d}$.}
\end{theorem}

\subsection{Contact graphs of unit sphere packings in $\mathbb{E}^{d}$}

Given the NP-hardness of recognizing contact graphs of unit sphere packings for $d=2, 3, 4$, Hlin\v{e}n\'{y} \cite{Hl} conjectured that the problem remains NP-hard in any fixed dimension.

\begin{conjecture}\label{hard}
The recognition of contact graphs of unit sphere packings is NP-hard in any fixed dimension $d\geq 2$.
\end{conjecture}

Hlin\v{e}n\'{y} and Kratochvil \cite{HlKr} made some progress towards this conjecture. They reproved Theorem \ref{3-hard} using the rather elaborate notion of a \textit{scheme of an $m$-comb} and then proved Conjecture \ref{hard} for $d=8, 24$. To define an $m$-comb we need to introduce some more terminology.

For a hyperplane $h$ in ${\mathbb{E}}^{d}$ and $S\subseteq {\mathbb{E}}^{d}$, let $S/h$ denote the mirror reflection of $S$ across $h$. We say that a set $S$ is a minimal-distance representation of a graph $G$, denoted by $G=M(S)$, if the vertices of $G$ are the points of $S$, and the edges of $G$ correspond to minimal-distance pairs of points in $S$. The graph $G$ is then called the \textit{minimal-distance graph of $S$}. Also, let $m(S)$ denote the minimal distance among pairs of points of $S$. Finally, when $m(S)=1$, we say that the set $S$ is \textit{rigid} in ${\mathbb{E}}^{d}$ if for any set $S' \subseteq {\mathbb{E}}^{d}$, $m(S')=1$, the following holds: If $\phi:M(S') \to M(S)$ is an isomorphism, then $\phi$ is an isometry of the underlying sets $S',S$. (Notice that the definition of a rigid set is slightly stronger than just saying that $S$ has a unique representation up to isometry.) For example, the vertices of a regular tetrahedron or a regular octahedron form rigid sets in ${\mathbb{E}}^{3}$. In general, the vertices of a $d$-dimensional simplex or a $d$-dimensional cross-polytope are rigid sets in ${\mathbb{E}}^{d}$. 

\begin{definition}[Scheme of an $m$-comb \cite{HlKr}]
Let $T, V, W$ be point sets in ${\mathbb{E}}^{d}$, and let $\boldsymbol{\alpha}, \boldsymbol{\beta}$ be vectors in ${\mathbb{E}}^{d}$. The five-tuple $(V,W,T,\boldsymbol{\alpha}, \boldsymbol{\beta})$ is called a scheme of an $m$-comb in ${\mathbb{E}}^{d}$ if the following conditions are satisfied: 
\begin{itemize}
\item The sets $V \cup W$ and $T$ are both rigid in ${\mathbb{E}}^{d}$, and $m(V)=m(V \cup W)=m(T)=1$. 
\item The set $V$ spans a hyperplane $h$ in ${\mathbb{E}}^{d}$. The vector $\boldsymbol{\alpha}$ is parallel to $h$. Let $T_{0} =T \cap (T -\boldsymbol{\beta})$. Then the set $T_0$ spans the whole ${\mathbb{E}}^{d}$. For $i=0, \ldots, m-1$, the set $(T_{0} + i\boldsymbol{\alpha}) \cap V$ spans the hyperplane $h$. 
\item Let $c$ be the maximal distance of $W$ from $h$. Then the distance between $h$ and $h+\boldsymbol{\beta}$ is greater than $2c + 1$. The distance between $h$, and $T + \boldsymbol{\beta}$ or $T - 2\boldsymbol{\beta}$, is greater than $c+1$.
\item Let $p$ be the straight line parallel to $\boldsymbol{\beta}$ such that the maximal distance $c'$ between $p$ and the points of $T$ is minimized. Then the distance of $p$ and $p +\boldsymbol{\alpha}$ is greater than $2c' +1$. For $j \in {\mathbb{Z}} - \{0,\ldots,m -1\}$, the distance between the sets $V\cup W$ and $p+j\boldsymbol{\alpha}$ is greater than $c'+1$.
\item The sets $T$ and $(T - \boldsymbol{\beta})\setminus T$ are non-overlapping, and the sets $T$ and $T + 2\boldsymbol{\beta}$ are strictly non-overlapping; while the sets $T$ and $(T/h) + \boldsymbol{\beta}$ are overlapping each other. Let $T'=T \cup (T - \boldsymbol{\beta})$. Then, for $i=0,\ldots, m-1$, the sets $V$ and $(T' +i\boldsymbol{\alpha})\setminus V$ are non-overlapping. 
\end{itemize}
\end{definition}  

The term `$m$-comb' comes from the actual geometry of such a scheme, which is comb-like (see the illustration of an $m$-comb in \cite{HlKr}). It turns out that if $d\geq 3$ is such that for every $m> 0$, there exists a scheme of an $m$-comb in ${\mathbb{E}}^{d}$, then the recognition of contact graphs of unit sphere packings in ${\mathbb{E}}^{d}$ is an NP-hard problem \cite{HlKr}. 

\begin{theorem}\label{nphard}
The problem of recognizing contact graphs of unit sphere packings is NP-hard in ${\mathbb{E}}^{3}$, ${\mathbb{E}}^{4}$, ${\mathbb{E}}^{8}$ and ${\mathbb{E}}^{24}$. 
\end{theorem}

The proof relies on constructing such schemes for $d=3,4,8,24$. For $d\neq 2,3,4,8,24$, the complexity of recognizing unit sphere contact graphs is unknown, while for $d=2$ it is NP-hard from Theorem \ref{2-hard}.

\subsection{Digital and totally separable sphere packings in $\mathbb{E}^{d}$}

Let us imagine that we generate totally separable packings of unit diameter balls in $\mathbb{E}^{d}$ such that every center of the balls chosen, is a lattice point of the integer lattice $\mathbb{Z}^{d}$ in $\mathbb{E}^{d}$. Then, as in Section \ref{sec:polyomino}, let $c_{\mathbb{Z}}(n,d)$ denote the largest possible contact number of all packings of $n$ unit diameter balls obtained in this way.

\begin{theorem}\label{dD-cubic}
$c_{\mathbb{Z}}(n,d)\le \lfloor dn- d n^{\frac{d-1}{d}} \rfloor $, for all $n>1$ and $d\ge 2$.
\end{theorem}

For the convenience of the reader, we recall here the elementary short proof of Theorem~\ref{dD-cubic} from \cite{BeSzSz}. A union of finitely many axis parallel $d$-dimensional orthogonal boxes having pairwise disjoint interiors in $\mathbb{E}^{d}$ is called a {\it box-polytope}. One may call the following statement the isoperimetric inequality for box-polytopes, which together with its proof presented below is an analogue of the isoperimetric inequality for convex bodies derived from the Brunn--Minkowski inequality. (For more details on the latter see for example, \cite{Ba97}.) 

\begin{lemma}\label{isoperimetric-box-polytopes}
Among box-polytopes of given volume the cubes have the least surface volume.
\end{lemma}

\begin{proof}  Without loss of generality, we may assume that the volume ${\rm vol}_d(\mathbf{A})$ of the given box-polytope $\mathbf{A}$
in $\mathbb{E}^{d}$ is equal to $2^d$, i.e., ${\rm vol}_d(\mathbf{A})=2^d$. Let $\mathbf{C}^d$ be an axis parallel $d$-dimensional cube of $\mathbb{E}^{d}$ with ${\rm vol}_d(\mathbf{C}^d)=2^d$. Let the surface volume of $\mathbf{C}^d$ be denoted by ${\rm svol}_{d-1}(\mathbf{C}^d)$. Clearly, ${\rm svol}_{d-1}(\mathbf{C}^d)=d\cdot{\rm vol}_d(\mathbf{C}^d)$. On the other hand,
if ${\rm svol}_{d-1}(\mathbf{A})$ denotes the surface volume of the box-polytope $\mathbf{A}$, then it is rather straightforward to show that

$${\rm svol}_{d-1}(\mathbf{A})=\lim_{ \epsilon\to 0^+}\frac{{\rm vol}_d(\mathbf{A}+\epsilon\mathbf{C}^d)-{\rm vol}_d(\mathbf{A})}{\epsilon}\ ,$$
where $"+"$ in the numerator stands for the Minkowski addition of the given sets. Using the Brunn--Minkowski inequality (\cite{Ba97}) we get that
$${\rm vol}_d(\mathbf{A}+\epsilon\mathbf{C}^d)\ge \left( {\rm vol}_d(\mathbf{A})^{\frac{1}{d}}+{\rm vol}_d(\epsilon\mathbf{C}^d)^{\frac{1}{d}}\right)^d= \left( {\rm vol}_d(\mathbf{A})^{\frac{1}{d}}+\epsilon\cdot {\rm vol}_d(\mathbf{C}^d)^{\frac{1}{d}}\right)^d.$$
Hence,
\begin{align*}{\rm vol}_d(\mathbf{A}+\epsilon\mathbf{C}^d) &\ge {\rm vol}_d(\mathbf{A})+d\cdot {\rm vol}_d(\mathbf{A})^{\frac{d-1}{d}}\cdot\epsilon \cdot {\rm vol}_d(\mathbf{C}^d)^{\frac{1}{d}}\\
&= {\rm vol}_d(\mathbf{A})+\epsilon \cdot d \cdot {\rm vol}_d(\mathbf{C}^d)\\
&= {\rm vol}_d(\mathbf{A})+\epsilon\cdot {\rm svol}_{d-1}(\mathbf{C}^d)\ .  
\end{align*}
So,
$$\frac{{\rm vol}_d(\mathbf{A}+\epsilon\mathbf{C}^d)-{\rm vol}_d(\mathbf{A})}{\epsilon}\ge {\rm svol}_{d-1}(\mathbf{C}^d)$$
and therefore, ${\rm svol}_{d-1}(\mathbf{A})\ge {\rm svol}_{d-1}(\mathbf{C}^d)$, finishing the proof of Lemma~\ref{isoperimetric-box-polytopes}. 
\end{proof}

\begin{corollary}\label{iso-box-poly}
For any box-polytope $\mathbf{P}$ of $\mathbb{E}^{d}$ the isoperimetric quotient of $\mathbf{P}$
is at least as large as the isoperimetric quotient of a cube, i.e., 
$$\frac{{\rm svol}_{d-1}(\mathbf{P})^{d}}{{\rm vol}_d(\mathbf{P})^{d-1}}\ge (2d)^d\ .$$
\end{corollary}

Now, let $\overline{{\cal P}}:=\{\mathbf{c}_1+\overline{\mathbf{B}}^d, \mathbf{c}_2+\overline{\mathbf{B}}^d, \dots , \mathbf{c}_n+\overline{\mathbf{B}}^d\} $ denote the totally separable packing of $n$ unit diameter balls with centers $\{\mathbf{c}_1, \mathbf{c}_2, \dots , \mathbf{c}_n\}\subset \mathbb{Z}^{d}$ having contact number $c_{\mathbb{Z}}(n,d)$ in $\mathbb{E}^d$. ($\overline{{\cal P}}$ might not be uniquely determined up to congruence in which case $\overline{{\cal P}}$ stands for any of those extremal packings.) Let $\mathbf{U}^d$ be the axis parallel $d$-dimensional unit cube centered at the origin $\mathbf{o}$ in $\mathbb{E}^d$. Then the unit cubes $\{\mathbf{c}_1+\mathbf{U}^d, \mathbf{c}_2+\mathbf{U}^d, \dots , \mathbf{c}_n+\mathbf{U}^d\}$ have pairwise disjoint interiors and $\mathbf{P}=\cup_{i=1}^{n} (\mathbf{c}_i+\mathbf{U}^d)$ is a box-polytope. Clearly, ${\rm svol}_{d-1}(\mathbf{P})=2dn-2c_{\mathbb{Z}}(n,d)$. Hence,
Corollary~\ref{iso-box-poly} implies that
$$2dn-2c_{\mathbb{Z}}(n,d)={\rm svol}_{d-1}(\mathbf{P})  \ge 2d {\rm vol}_d(\mathbf{P})^{\frac{d-1}{d}}=  2dn^{\frac{d-1}{d}}\ .$$
So, $dn-dn^{\frac{d-1}{d}}\ge c_{\mathbb{Z}}(n,d)$, finishing the proof of Theorem~\ref{dD-cubic}.

Here we recall Theorem \ref{polyomino} and refer to \cite{BeSzSz} to note that the upper bound of Theorem~\ref{dD-cubic} is sharp for $d=2$ and all $n>1$ and for $d\ge 3$ and all $n=k^d$ with $k>1$. On the other hand, it is not a sharp estimate for example, for $d=3$ and $n=5$.  

We close this section by stating the recent upper bounds of \cite{BeSzSz} for the contact numbers of totally separable unit ball packings in $\Eu^{d}$. 

\begin{theorem}\label{dD}
$c_{\rm sep}(n,d)\le dn-\frac{1}{2d^{\frac{d-1}{2}}}n^{\frac{d-1}{d}} $, for all $n>1$ and $d\ge 4$. 
\end{theorem}

\section{Contact graphs of non-congruent sphere packings}\label{sec:noncongruent}
So far, we have exclusively focused on contact graphs of packings of congruent spheres. In this section, we discuss what is known for general non-congruent sphere packings. Let us denote by $c^{*}(n,d)$ the maximal number of edges in a contact graph of $n$ not necessarily congruent $d$-dimensional balls. Clearly, $c^{*}(n,d)\geq c(n,d)$, for any positive integers $n$ and $d\geq 2$. 

The planar case was first resolved by Koebe \cite{Ko} in 1936. Koebe's result was later rediscovered by Andreev \cite{An} in 1970 and by Thurston in \cite{Th} 1978.\footnote{It is worth-noting that Koebe's paper was written in German and titled `Kontaktprobleme der konformen Abbildung' (Contact problems of conformal mapping). Andreev's paper appeared in Russian. Probably, the first instance of this result appearing in English was in Thurston's lecture notes that were distributed by the Princeton University in 1980. However, the lectures were delivered in 1978-79 \cite{Th}.} The result is referred to as Koebe--Andreev--Thurston theorem or the circle packing theorem.

In terms of contact graphs, the result can be stated as under. 

\begin{theorem}[Koebe--Andreev--Thurston]\label{koebe}
A graph $G$ is a contact graph of a (not necessarily congruent) circle packing in ${\mathbb{E}}^{2}$ if and only if $G$ is planar. 
\end{theorem}

In other words, for any planar graph $G$ of any order $n$, there exist $n$ circular disks with possibly different radii such that when these disks are placed with their centers at the vertices of the graph, the disks centered at the end vertices of each edge of $G$ touch. In addition, this cannot be achieved for any nonplanar graph. This is a rather unique result that is, as we will see shortly, highly unlikely to have an analogue in higher dimensions. It shows that $c^{*}(n,2)=3n-6$, for $n\geq 2$, which is the number of edges in a maximal planar graph. 

A similar simple characterization of contact graphs of general not necessarily congruent sphere packings cannot be found for all dimensions $d\geq 3$, unless P = NP. We briefly discuss this here. In \cite{Hl, HlKr}, the authors report that Kirkpatrick and Rote informed them of the following result in a personal communication in 1997. The proof appears in \cite{HlKr}. 

\begin{theorem}\label{connection}
A graph $G$ has a $d$-unit-ball contact representation if and only if the graph $G\oplus K_2$ has a $(d+1)$-ball contact representation. 
\end{theorem} 

Here $K_2$ denotes the complete graph on two vertices, while $G\oplus H$ represents the graph formed by taking the disjoint union of $G$ and $H$ and then adding all edges across \cite{HlKr}. Theorem \ref{connection} provides an interesting connection between contact graphs of unit sphere packings in ${\mathbb{E}}^{d}$ and contact graphs of not necessarily congruent sphere packings in ${\mathbb{E}}^{d+1}$. Combining this with Theorem \ref{2-hard}, \ref{3-hard} and \ref{nphard} gives the following \cite{HlKr}.

\begin{corollary}
The problem of recognizing general contact graphs of (not necessarily congruent) sphere packings is NP-hard in dimensions $d=3,4,5,9,25$.
\end{corollary}
 
Not much is known about $c^{*}(n,d)$, for $d\geq 4$. However, for $d=3$, an upper bound was found by Kuperberg and Schramm \cite{KuSc}. (Also see \cite{HaSzUj} for some elementary results on forbidden subgraphs of contact graphs of non-congruent sphere packings in ${\mathbb{E}}^{3}$.) Define the average kissing number $k^{*}_{\rm av}(d)$ in dimension $d$ as the supremum of average vertex degrees among all contact graphs of finite sphere packings in ${\mathbb{E}}^{d}$. In a packing of three dimensional congruent spheres, a sphere can touch at the most 12 others \cite{ScWa}. Thus a three dimensional ball $B$ cannot touch more than 12 other balls at least as large as $B$. It follows that $k^{*}_{\rm av}(3)\leq 2k(3) = 24$. In \cite{KuSc}, this was improved to $12.566 \approx 666/53 \leq k^{*}_{\rm av}(3)< 8 + 4\sqrt{3} \approx 14.928$. In the language of contact numbers, the Kuperberg--Schramm bound translates into the following.   

\begin{theorem}\label{kuperberg-schramm}
$c^{*}(n,3) < (4+2 \sqrt{3})n \approx 7.464n .$
\end{theorem}

The method of Kuperberg and Schramm relies heavily on the geometry of 3-dimensional space. As a result it seems difficult to generalize it to higher dimensions. We close this section with the following open question. 

\begin{problem}
Find upper and lower bounds on $c^{*}(n, d)$ in the spirit of Kuperberg--Schramm bounds on $c^{*}(n, 3)$. 
\end{problem}

\section*{Acknowledgments}
 The first author is partially supported by a Natural Sciences and Engineering Research Council of Canada Discovery Grant. The second author is supported by a Vanier Canada Graduate Scholarship (NSERC), an Izaak Walton Killam Memorial Scholarship and Alberta Innovates Technology Futures (AITF). The authors
would like to thank the anonymous referee for careful reading and an interesting reference.

\bigskip

\normalsize
\noindent K\'aroly Bezdek \\
\small{Department of Mathematics and Statistics, University of Calgary, Canada}\\
\small{Department of Mathematics, University of Pannonia, Veszpr\'em, Hungary\\
\small{\texttt{E-mail:bezdek@math.ucalgary.ca}}

\normalsize

\bigskip
\noindent and
\bigskip

\noindent Muhammad A. Khan \\
 \small{Department of Mathematics and Statistics, University of Calgary, Canada}\\
 \small{\texttt{E-mail:muhammkh@ucalgary.ca}}

\end{document}